\documentclass[reqno,b5paper]{amsart}
\usepackage{amssymb}
\usepackage{dsfont}
\usepackage{color}
\usepackage{graphicx}
\usepackage{float}

\newtheorem{theorem}{Theorem}[section]

\newtheorem{proposition}[theorem]{Proposition}
\newtheorem{corollary}[theorem]{Corollary}
\newtheorem{lemma}[theorem]{Lemma}

\newtheorem{remark}[theorem]{Remark}
\newtheorem{definition}[theorem]{Definition}
\newtheorem{fact}[theorem]{Fact}

\newcommand{\T}{\mathbb{T}}
\newcommand{\Z}{\mathbb{Z}}
\newcommand\A{\mathcal{A}}
\newcommand\iA{\mathcal{IA}}
\newcommand\D{\mathcal{D}}

\newcommand\fN{\mathcal{N}}
\newcommand\DD{\mathbf{D}}

\newcommand\NN{\mathbf{N}}
\newcommand\NZ{\mathbf{N_0}}
\newcommand\AR{\mathbf{A}}
\newcommand\wD{\mathbf{wD}}

\newcommand\iAR{\mathbf{IA}}
\newcommand\F{\mathcal{F}}
\newcommand{\N}{\mathbb{N}}

\newcommand{\supp}{\mathrm{supp}}

\catcode`\@=12
\def\T{{\mathbb T}}

\def\iI{\mathcal{I}}
\def\A{\mathcal{A}}
\def\F{\mathcal{F}}
\def\Z{{\mathbb Z}}
\def\N{{\mathbb N}}
\def\R{{\mathbb R}}
\def\Q{{\mathbb Q}}

\makeindex

\begin{document}
\title[On the role of dense ideals on $\iI$-characterized subgroups ]{On the role of dense ideals on $\iI$-characterized subgroups and associated $\iI$-Arbault sets}
%%\title[ Trigonometric thin sets and $s$-characterized subgroups ]{Trigonometric thin sets and $s$-characterized subgroups}
%% note on $s$-characterized subgroups for arithmetic sequences]
%%{A note on $s$-characterized subgroups for arithmetic sequences}
\author{Pratulananda Das}

\address{Department of Mathematics, Jadavpur University, Kolkata-700032, India}
\email {pratulananda@yahoo.co.in}

\author{Ayan Ghosh}
\address{Department of Mathematics, NIT Durgapur, Durgapur - 713209, West Bengal, India}
\email{ayanghosh.jumath@gmail.com}

\begin{abstract}
This article is a continuation of the work done in [Das et al., Annals of Pure and Applied Logic 174 (2023), 103289] on  $\iI$-characterized subgroups and concerns in addition the generalized version of Arbault sets, namely, $\iI$-Arbault sets that is generated from this collection.  A special class of ideals, namely, non-$snt$ ideals was introduced in 2023 ($snt$ stands for ``strongly non translation invariant") which was found to be effective in understanding ``Eggleston's dichotomy". Here we introduce another new class of ideals $\mathcal{H}$, a subclass of non-$snt$ ideals. We are able to establish a (new) necessary and sufficient condition for a $P$-ideal to be dense in Theorem 2.8 using the condition used to define the class  $\mathcal{H}$ from which it follows that $\mathcal{H}$ coincides with dense $P$-ideals. Furthermore the important fact is established that the class of non-$snt$ ideals strictly contains the class of dense $P$-ideals. Subsequent investigation  leads us to Theorem 2.14 which establishes the novelty of $\iI$-characterized subgroups for ideals coming from the class  $\mathcal{H}$ and arbitrary arithmetic sequences showing that these are not even Arbault sets. Before all these, we also present a correct and revised proof of the main result of [Das et al., Annals of Pure and Applied Logic 174 (2023), 103289], namely Theorem 3.8 (which had a gap in the proof). Finally we extend the results of [Das et al., Bul. Sci. Math. 179 (2022), 103157] for generalized version of Arbault sets and primarily show that for each ideal coming from the class  $\mathcal{H}$, the class of $\iI$-Arbault sets not only properly contains the class of classical Arbault sets but also a large subfamily of $\NN$-sets (also called ``sets of absolute convergence") while being contained in the class of weak Dirichlet sets.
%In particular it properly contains the family of $\NZ$-sets which have been extensively used in the literature (see \cite{Ar, Ka, Ko}). Though distinct from the class of $\NN$-sets, this happens to be a new class strictly lying between the class of Arbault sets and the class of weak Dirichlet sets.
%
%The motivation for thinking of such a possible class comes from the well known observation that the characterized subgroups of the circle group forms a basis of Arbault sets and the recent introduction of the notion of statistically characterized subgroups \cite{DDB} extending the notion of characterized subgroups. In Section 2 of the article it is shown that there are statistically characterized subgroups which can't be characterized by any sequence of integers establishing the ``novelty" of the notion. This naturally paves the way for a new class of sets generated by the class of statistically characterized subgroups as basis which we name statistical Arbault sets.
\end{abstract}
\maketitle
\section{Introduction and the background}

Throughout $\R$, $\Q$, $\Z$, $\T$ and $\N$ will stand for the set of all real numbers, the set of all rational numbers,
the set of all integers, the set of all complex numbers on the unit circle and the set of all positive natural numbers respectively. The first three are equipped with their usual abelian group structure and the circle group $\T=\R/\Z$ is equipped with the operation of addition mod 1. Following \cite{Ka1,Ka}, we may identify $\T$ with the interval [0,1] identifying 0 and 1. Any real valued function $f$ defined on $\T$ can be identified with a periodic function defined on the whole real line $\R$ with period 1, i.e., $f(x+1)=f(x)$ for every real $x$. When referring to a set $X\subseteq \T$ we assume that $X\subseteq [0,1]$ and $0\in X$ if and only if $1\in X$.\\

A non-empty family $\iI \subset$ $\mathcal{P}(\N)$ is called an ideal on $\N$ whenever
\begin{itemize}
\item $\N \notin \iI$,
\item if $A,B \in \iI$ then $A\cup B \in \iI$,
\item $A \subset B$ and $B \in \iI$ then $A \in \iI$.
\end{itemize}
An ideal $\iI$ is called admissible (or free) if $\{n\}\in \iI$ for all $n\in\N$. If $\iI$ is an ideal on $\N$, by $\iI^*$ we denote the dual
filter, that is, $\iI^* = \{\N\setminus A\colon A\in \iI\}$ whereas $\iI^+ = \{B\subseteq \N: B\notin \iI\}$ is called the coideal of the ideal $\mathcal{I}$. By $Fin$ we denote the ideal of all finite subsets of $\N$. All ideals henceforth will be assumed to be free ideals. One very prominent ideal (coming from number theory) is defined as follows.

For $m,n\in\mathbb{N}$ and $m\leq n$, let $[m, n]$ denotes the set $\{m, m+1, m+2,...,n\}$. By $|A|$ we denote the cardinality of a set $A$. The lower and the upper natural densities of $A \subset \mathbb{N}$ are defined by \cite{B1}
$$
\underline{d}(A)=\displaystyle{\liminf_{n\to\infty}}\frac{|A\cap [1,n]|}{n} ~~\mbox{and}~~
\overline{d}(A)=\displaystyle{\limsup_{n\to\infty}}\frac{|A\cap [1,n]|}{n}.
$$
If $\underline{d}(A)=\overline{d}(A)$, we say that the natural density of $A$ exists and it is denoted by $d(A)$. It is easy to check that $\iI_d = \{A \subseteq \N: d(A) = 0\}$ forms a free ideal (it is known that $\iI_d$ is also an analytic $P$-ideal just like $Fin$).

For two subsets $A, B$ of $\N$ and an ideal $\iI$, we will write
\begin{itemize}
\item[$\bullet$] $A\subseteq_\iI B$ if $A\subseteq B$ and $B\setminus A\in\iI$,
\item[$\bullet$] $A \subset \N$ is called $\iI$-translation invariant if $A+n=\{m+n \in\N: \ m\in A\}$ belongs to $\iI$ for all $n\in\Z$.
\end{itemize}
Note that for any $A\subseteq_\iI B$, $B \in \iI^+$ implies $A\in\iI^+$. $\iI$ is called translation invariant if every $A \in \iI$ is $\iI$-translation invariant and $\iI$ is called dense if for every $A\subseteq\N$ there exists $B\subseteq A$ such that $B\in\iI$.

An ideal $\iI$ on $\N$ is called a $P$-ideal if for every sequence $(A_{n})$ of sets in $\iI$ there is a set $A \in \iI$ such that $A_{n} \subset^*A$ for all $n \in \N$.  Further every ideal $\iI$ on $\N$ can be treated as a subset of the Cantor space $2^{\N}$ in view of the fact that $\mathcal{P}(\N)$ and $2^{\N}$ can be identified via the characteristic functions. An ideal $\iI$ is called analytic if it corresponds to an analytic subset of the Cantor space $2^{\N}$. A highly nontrivial theorem of Solecki (see \cite{So1}) states that each analytic $P$-ideal on $\N$ is actually the exhaustive set for some lower semicontinuous submeasure $\varphi$ on $\N$. Finally an ideal $\iI$ is called tall or dense if every infinite subset of $\N$ contains an infinite subset belonging to $\mathcal{I}$. Over the years dense ideals have been used and studied extensively (see for example \cite{H, H1, H2, AP, AZ}).

%A submeasure on $\N$ is a function $\varphi\colon {\mathcal
%P}(\N)\to [0,\infty]$ such that:
%\begin{itemize}
%\item $\varphi(\emptyset)=0$,
%\item if $A \subset B$ then $\varphi(A)\leq \varphi(B)$,
%\item $\varphi(A\cup B) \leq \varphi(A) + \varphi(B)$,
%\item $\varphi(\{n\})<\infty$ for all $n\in\mathbb{N}$.
%\end{itemize}
%A submeasure $\varphi$ is called a lower semicontinuos submeasure (in short, lscsm) if $\varphi(A)=\lim_{n\to\infty} \varphi(A\cap [1,n])$ for all $A \subset \N$ (this condition is equivalent to the classical lower semicontinuity of the function $\varphi\colon
%2^{\N} \to [0,\infty]$). For any lscsm $\varphi$, we consider the exhaustive ideal given by
%$$
%Exh(\varphi)=\{ A \subset \N\colon \lim_{n\to\infty}\varphi(A
%\setminus [1,n])=0\}.
%$$
%It follows that for every lscsm $\varphi$ on $\N$,
%$Exh(\varphi)$ is an $F_{\sigma \delta}$ $P$-ideal
%\cite[Lemma~1.2.2]{Far}. A highly nontrivial theorem of Solecki
%\cite{So} states that each analytic $P$-ideal on $\N$ is of the
%form $Exh(\varphi)$ for some lscsm $\varphi$ on $\N$.

Next we come to the notion of convergence defined using the notion of ideals. The following notion first appeared in the work of celebrated mathematician Henry Cartan \cite{Ca1} and then again reappeared in 2000.

\textbf{Definition} (cf. \cite{KSW})
In a topological space $X$, given an ideal $\iI$, we say that a sequence $(x_n)$ is $\iI$-convergent to $x \in X$ whenever for every open set $U$ containing $x$, the set $\{n\in\N: x_n\notin U\}\in \iI$ (we will write $x_n \rightarrow x$ w.r.t $\iI$).

It is known (see \cite{KSW}) that ``$(x_n)_{n\in\N}$ is $\iI$-convergent to $x$ $\Longleftrightarrow$ there is a set $M \in \iI^*$ such that $(x_n)_{n \in M}$ is convergent to $x$" iff $\iI$ is a $P$-ideal. In the literature about ideal convergence, a very prominent role has been played by the class of analytic $P$-ideals (for example see \cite{KSW, LR} and the survey article \cite{FNS}) while these ideals had long been topics of much interest in set theory (see the excellent survey article \cite{H2} where many references can be found).

Using the notion of ideal convergence, in \cite{DG6} the ideal version of characterized subgroups (Definition \ref{defIchar}), namely, $\iI$-characterized subgroups was introduced which encompasses all such notions existing in the literature. A brief history with all necessary definitions about characterized subgroups has been given at the beginning of Section 2. For more investigations on $\iI$-characterized subgroups one can see more recent articles \cite{DI2, DI3}. However in \cite{DG6} as also subsequent articles, the novelty of the notion was not studied, i.e., whether it is possible to point out a class of ideals with a suitable property which can ensure that for any arithmetic sequence, the associated $\iI$-characterized subgroup can not be classically generated from any sequence. \\

Our secondary interest in this article is Arbault sets, a particular type of trigonometric thin sets. These sets came into prominence initially for their association with trigonometric series \cite{Ca1} and then Fourier analysis.
%A series of the form
%$$
%\frac{c_0}{2}+\sum\limits_{n=1}^\infty c_n\cos 2\pi nx + d_n\sin 2\pi nx
%$$
%is called a trigonometric series where $c_{n-1},d_n\in\R$ for all $n\in\N$. In \cite{Ca1}, G. Cantor had shown the first uniqueness result:
%
%If a trigonometric series converges to zero for all $x\in [0,1]$ then all its coefficients get vanished.

One can have a nice and brief overview of the initial progress, several modifications and generalizations in this direction from \cite{Lu,Yo} and also the survey articles \cite{BKR,LP} (for a general view) and one should note that though historically trigonometric thin sets came into play while trying to understand the seemingly ``bad" sets beyond which absolute convergence happens, precisely the families of exceptional sets considered in trigonometric series theory or Fourier analysis, they themselves became objects of serious investigation over the years.

A more detailed discussion of these sets is given at the beginning of Section 3. In our article our main point of interest in this direction is the class of Arbault sets because they are actually generated from the class of characterized subgroups as a basis. Specifically A set $X\subseteq [0,1]$ is an $\AR$-set (Arbault set) if there exists an increasing sequence of naturals $(a_n)$ such that $\|a_nx\|$ converges pointwise to 0 on $X$.

Very recently in \cite{DGT} we had considered the class of trigonometric thin sets formed by statistically characterized subgroups \cite{DDB} (which is a special case of $\iI$-characterized subgroups for $\iI = \iI_d$) as basis, which we named, statistical Arbault sets and showed that it forms a new class of thin sets properly containing the class of classical Arbault sets among several interesting results.\\

%Taking in particular $\iI = Fin$ one obtains the classical notion of characterized subgroups (Definition \ref{defchar}) which generates the class of Arbault sets as a basis while for the ideal $\iI_d$, the notion of statistically characterized subgroups (Definition \ref{defschar}) comes along \cite{DDB} which in turn generates the class of statistical Arbault sets as basis \cite{DGT} (a little more detailed discussion about characterized subgroups is given in Section 2 in order to make the article more accessible to readers).

In this article we start with $\iI$-characterized subgroups and the first thing we do is to present a correct proof of the main result of \cite{DG6}, namely Theorem 3.8 as there happened to be a gap in the argument of that result. 

Next just as we had introduced the notion of $snt$-ideals (see Definition \ref{defsnt}) in \cite{DG6} to understand which property ensures ``Eggleston's Dichotomy", here we introduce the following:

\begin{eqnarray*}
(\natural) \quad \text{For every infinite } A \subseteq \N, \text{ there exists an infinite } B \subseteq A \\ \text{ such that } B \in \iI \text{ and } B \text{ is } \iI\text{-translation invariant}.
\end{eqnarray*}

Accordingly, we define the class of ideals
\[
\mathcal{H} = \bigl\{\, \iI \subseteq \mathcal{P}(\N) : \iI \text{ is a free $P$-ideal and satisfies property } (\natural) \,\bigr\}.
\]

Next we obtain a rather surprising observation in Theorem \ref{thdense01} which shows that a $P$-ideal $\iI$ is dense if and only if it satisfies the property $(\natural)$. As far as literature is concerned this seems to be a new characterization of density property. Subsequently certain examples are given in order to understand the comparison of the above property vis a vis ``translation invariant". Finally we come to the understanding of the novelty of the notion of $\iI$-characterized subgroups, as promised, and it turns out that for the ideals belonging to the class $\mathcal{H}$, the $\iI$-characterized subgroups generated by any arithmetic sequence can not be characterized in the classical way (see Theorem \ref{th6}).\\

In the last section we consider the ideal version of Arbault sets, $\iI$-Arbault sets, which can be generated from the $\iI$-characterized subgroups in the natural way and it turns out, in a very non-trivial way, that for every ideal coming from the class of dense $P$-ideals, the resulting class of $\iI$-Arbault sets indeed forms a new class of trigonometric thin sets properly containing the class of classical Arbault sets. The other advantage is that unlike the class of Arbault sets, this new class is much bigger and actually contains a large subfamily of $\NN$-sets, in particular, it contains types of $\NN$-sets which have been extensively used in Fourier analysis. In this particular respect this new class seems much more beneficial than the class of Arbault sets which is only known to contain $\NZ$-sets.
%automatically encompasses Arbault sets and its generalization, statisticlly Arbault sets as special cases. However as it happens quite often, if one considers any arbitrary ideal, then the generalization may not yield any advantage, while on the contrary some of the existing results may break down. As has been the case (again one may see \cite{KSW, LR} as examples of case study) when it comes to generalization using ideals, the most important fact that one has to keep in mind is "what type of ideals to consider" in order to obtain the desired results. Here our reference point is the ideal $\iI_d$ and the observation that it always produces a strictly bigger class of thin sets properly containing the class of Arbault sets. As a ``suitable class of ideals" containing the ideal $\iI_d$ one natural choice is the class of non-$snt$ ideals (where $snt$ stands for strongly non-translation invariant) but it is not at all obvious whether every member of this class would have the desired property.

 As a concluding remark to the introduction, an important takeaway is the following observation: In the rich theory of characterized subgroups (see \cite{DDG}) it has been known that every countable subgroup can be characterized but for uncountable Borel subgroups there have not been much clarity as also deficiency of examples of uncountable Borel subgroups which can not be characterized. In view of Theorem \ref{th6} one can conclude that there is no dearth of such examples.
%
%
% In Section 2 of the article it is shown that there are statistically characterized subgroups which can't be characterized by any sequence of integers establishing the ``novelty" of the notion which was missing when the notion of statistically characterized subgroups was introduced in \cite{DDB}. This naturally paves the way for a new class of sets generated by the class of statistically characterized subgroups as basis which we name statistical Arbault sets which is introduced in Section 3 where some basic properties are established. Finally Section 4 is devoted to the comparison of this new class with the existing classes of thin sets.
%
\section{The novelty of $\iI$-characterized subgroups and the leading role played by dense ideals}
Let us begin with the classical theory as to how one can generate historically nice subgroups of the circle via the notion of usual convergence (we would like to think it as $Fin$-convergence). These subgroups were first considered in connection with distribution of sequences of multiples of a given real number mod 1 \cite{W}. From another direction, the notion of characterized subgroups has evolved over the years as a generalization of the notion of  torsion subgroup (recall that an element $x$ of an abelian group is torsion if there exists $k \in \mathbb{N}$ such that $kx = 0$).
%An element $x$ of an abelian topological group $G$ is  \cite{B}:
%\begin{itemize}
%\item[(i)] {\em topologically torsion} if $n!x \rightarrow 0;$
%\item[(ii)] {\em topologically $p$-torsion}, for a prime $p$, if $p^nx \rightarrow 0.$
%\end{itemize}

%It is obvious that any $p$-torsion element is topologically $p$-torsion. Armacost \cite{A} defined the subgroups
%$$ X_p = \{x \in X: p^nx \rightarrow 0\} \ ~\mbox{and} ~ \ X! =  \{x \in X: n!x \rightarrow 0\}$$
%of an abelian topological group $X$,  started their investigation whose answers were partially given by Arbault himself and completed later by Borel \cite{Bo2}.

Thinking $k$ as a constant sequence $(k)$ one can extend the notion by considering a sequence of integers which gives rise to the notion of topologically torsion elements and corresponding subgroup. We now recall the definition of, so as to speak, a ``classical" characterized subgroup  of $\T$.

\begin{definition}\label{defchar}
Let $(a_n)$ be a sequence of integers, the subgroup
$$
t_{(a_n)}(\T) := \{x\in \T: a_nx \to 0\mbox{ in } \T\}.
$$
of $\T$ is called {\em a characterized} $($by $(a_n))$ {\em subgroup} of $\T$ while such an element $x$ is generally called topologically torsion elements of $\T$.
\end{definition}

Even if the notion was inspired by the various (earlier)  instances (see \cite{DDB} for the details), the term ``characterized" appeared much later, coined in \cite{BDS} and since then, have been of much interest in different areas of Mathematics (one must see the excellent survey article \cite{DDG} where the rich history along with the results and references can be found).

%     The notion of natural density which is central to this article which we now present below. For $m,n\in\mathbb{N}$ and $m\leq n$, let $[m, n]$ denotes the set $\{m, m+1, m+2,...,n\}$. By $|A|$ we denote the cardinality of a set $A$. The lower and the upper natural densities of $A \subseteq \mathbb{N}$ are defined by \cite{B1}
%$$
%\underline{d}(A)=\displaystyle{\liminf_{n\to\infty}}\frac{|A\cap [1,n]|}{n} ~~\mbox{and}~~
%\overline{d}(A)=\displaystyle{\limsup_{n\to\infty}}\frac{|A\cap [1,n]|}{n}.
%$$
%
%If $\underline{d}(A)=\overline{d}(A)$, we say that the natural density of $A$ exists and it is denoted by $d(A)$. An infinite sequence $(a_n)$ (or, an infinite subset of $\N$) is called lacunary if
%$$
% a_{n+1}-a_n \rightarrow \ \infty \ \forall\ n\in\N.
%$$
%
%Note that any lacunary subset of $\N$ has natural density zero. The idea of natural density was later used to define the notion of statistical convergence (\cite{F,St}, see also \cite{Fr,S}).

%\begin{definition}\label{Def1}\cite{DDB}
%A sequence $(x_n)$ in $\T$ is said to converge to $x_0\in\T$ statistically if for any $\eps > 0$, $d(\{n \in \mathbb{N}: \|x_n - x_0\| \geq \eps\}) = 0$.
%\end{definition}
%
%
%
The idea of, so as to speak, ``differently" characterized subgroups came into existence in \cite{DDB} to understand how the characterized subgroups behave when usual convergence (i.e., $Fin$-convergence) is replaced by a more general notion of convergence with respect to an ideal. For the starting, the ideal $\iI_d$ was considered and as a consequence using $\iI_d$-convergence (called statistical convergence in the literature) the statistically characterized subgroups of the circle were introduced in a similar fashion.
\begin{definition}\label{defschar}\cite{DDB}
For a sequence of integers $(a_n)$, the subgroup
$$
t^s_{(a_n)}(T) := \{x\in\T : a_nx\to 0 \mbox{ statistically in } \T\}
$$	
of $\T$ is called a statistically characterized (shortly $s$-characterized) (by $(a_n)$) subgroup of $\T$.
\end{definition}

\begin{corollary}\label{lfn1}
If  $\ d(M)=1$ then for any sequence of integers $(a_n)$,  $t^s_{(a_n)_{n\in \N}}(\T) =  t^s_{(a_n)_{n\in M}}(\T)$.
\end{corollary}
The result is obvious in view of the classical observation that a sequence of real numbers $(x_n)$ converges to a real number $x_0$ statistically if and only if there exists a set $M\in\mathcal{I}_d^\ast$ such that $(x_n)_{n\in M}$ usually converges to $x_0$.

Recall that an increasing sequence of naturals $(u_n)$ is called an arithmetic sequence if $u_n$ divides $u_{n+1}$ for all $n\in\N$. An arithmetic sequence $(u_n)$ is called $q$-bounded ($q$-divergent) if the sequence of ratios $(q_n)$ defined by $q_n=\frac{u_n}{u_{n-1}}$ is bounded (divergent).

In \cite{DDB} while developing the theory of $s$-characterized subgroups it was established that for any arithmetic sequence $(a_n)$, the $s$-characterized subgroup $t^s_{(a_n)}(\T)$ is always Borel, of size $\mathfrak c$ and strictly larger than the corresponding characterized subgroup $t_{(a_n)}(\T)$ (even when $t_{(a_n)}(\T)$ is uncountable). However this did not imply that using the notion of $s$-characterized subgroups one could actually obtain a new subgroup which can never be generated as in Definition 2.1., i.e., it can't be characterized by any sequence of naturals. In \cite[Theorem 2.7]{DGT} it was shown that one can indeed generate ``new" subgroups following the process of \cite{DDB}.

As a natural progression, $\iI$-characterized subgroups were considered for analytic $P$-ideals very recently in \cite{DG6} by considering $\iI$ convergence in the above definitions.
\emph{\begin{definition}\label{defIchar}
For a sequence of integers $(a_n)$, the subgroup
\begin{equation}\label{def:ideal:conv}
t^{\iI}_{(a_n)}(\T) := \{x\in \T: a_nx \to 0\  \mbox{ w.r.t }  \iI \mbox{ in } \T\}
\end{equation}
of $\T$ is called the subgroup of $\T$, $\mathcal{I}$-characterized by $(a_n)$ (shortly, an $\iI$-characterized subgroup).
\end{definition}}

Furthermore it has been identified in \cite{DG6} that there is a broad class of ideals called ``non-snt" ideals (containing the ideal $\iI_d$ and actually several classes of ideals from the literature) which have a profound impact on the behavior of the generated subgroups.

The first question that comes to mind when ideals come into picture is whether one can identify appropriate classes of ideals for which one would be able to generate new subgroups which can't be characterized by any sequence of naturals. In this section we show that class of ``non-snt" ideals with a special property do fulfil that role. It should be noted that in the absence of a density function the line of proof can not be replicated from the proof of \cite[Theorem 2.7]{DGT} and needs much more ingenuities.

\emph{\begin{definition}\label{defsnt}\cite{DG6}
An ideal $\iI$ will be called strongly non-translation invariant (in short $snt$-ideal) if it does not contain any infinite $\iI$-translation invariant set, i.e., for each infinite subset $A \in \iI$, there exists at least one $k \in \Z$ for which $A-k\not\in \iI$.
\end{definition}}
Before moving to the main result we just recall some basic facts and results from the literature which will be necessary to have a better understanding of Theorem \ref{th6}.
\begin{fact}\cite{DI1}
For any arithmetic sequence $(u_n)$ and $x\in [0,1]$, there exists a unique sequence of integers $(c_n(x))$, where $0\leq c_n(x)<q_n$, such that
\begin{equation}\label{canrep}
   x=\sum\limits_{n=1}^{\infty}\frac{c_n(x)}{u_n}
\end{equation}
and $c_n(x)<q_n-1$ for infinitely many $n$.
\end{fact}
%
%\begin{proof}
%For clarity we recall the construction of the sequence $(c_n)$. Take $c_1=\lfloor a_1x\rfloor$, where $\lfloor \ \rfloor$ denotes the integer part. Evidently $x-\frac{c_1}{a_1}<\frac{1}{a_1}$.
%
%Suppose, $c_1,c_2,\ldots,c_k$ are defined for some $k\geq 1$ with $x_k=\sum\limits_{n=1}^{k}\frac{c_n}{a_n}$ and $x-x_k<\frac{1}{a_k}$. Then the $(k+1)$-th element is defined as $c_{k+1}=\lfloor a_{k+1}(x-x_k)\rfloor$.
%\end{proof}
%%%%%%%%%%%%%%%%%%%%%%%%%%%%%%%%%%%%%%%%%%%%%%%%%%%%%%%%%%%%%%%%%%%%%%%%%%%%%%%%%%%%%%%%%%%%%%%%%%%%%%%
We would write $c_n$ in place of $c_n(x)$ when there is no confusion about $x$. For $x\in [0,1]$ with canonical representation (\ref{canrep}), we define $$\supp_{(u_n)}(x)=\{n\in\N\ : \ c_n(x)\neq0\}.$$
$$\supp^q_{(u_n)}(x)=\{n\in\N\ : \ c_n(x)= q_n-1\}.$$
In \cite{DG6}, it was established that snt ideals do not generate larger subgroups than the classical one for $q$-bounded arithmetic sequences. While the result itself is correct, the  proof given thereof contains a minor gap. Specifically, it was claimed that for any $x \in \T$, the support $supp(x)$ cannot be co-finite, which is in fact may not be true. Below we provide a revised and correct proof of \cite[Theorem 3.8]{DG6}.

%%%%%%%%%%%%%%%%%%%%%%%%%%%%%%%%%%%%%%%%%%%%%%%%%%%%%%%%%%%%%%%%%%%%%%%%%%%%%%%%%%%%%%%%%%%%%%%%%%%%%%%%%%%%%%%
%\begin{theorem}\label{newcharsub}
%For any non-$snt$ analytic $P$-ideal $\iI$ and for any arithmetic sequence $(a_n), \ ~ \ t^{\iI}_{(a_n)}(\T) \supsetneq t_{(a_n)}(\T) \ ~ \  \mbox{and} \ ~ \ |t^{\iI}_{(a_n)}(\T) \setminus t_{(a_n)}(\T) | = \mathfrak c$.
%\end{theorem}
%%%%%%%%%%%%%%%%%%%%%%%%%%%%%%%%%%%%%%%%%%%%%%%%%%%%%%%%%%%%%%%%%%%%%%%%%%%%%%%%%%%%%%%%%%%%%%%%%%%%%%%%%%%%%%%%
\begin{theorem}\label{Ifinex}
For a $q$-bounded arithmetic sequence $(a_n)$ and an analytic $P$-ideal $\iI$, $\ t^\iI_{(a_n)}(\T) = t_{(a_n)}(\T)$ if and only if $\iI$ is a $snt$-ideal.
\end{theorem}
\begin{proof}
Let $(a_n)$ be a $q$-bounded arithmetic sequence and $\iI$ be an analytic $P$-ideal. If $\iI$ is not a $snt$-ideal then $\ t^\iI_{(a_n)}(\T) \neq t_{(a_n)}(\T)$ by \cite[Theorem 3.6]{DG6}.

To prove the sufficiency part, let $\iI$ be an $snt$ analytic $P$-ideal. Clearly, we always have $t_{(a_n)}(\T) \subseteq t^\iI_{(a_n)}(\T)$. Suppose, for contradiction, that there exists $x \in t^\iI_{(a_n)}(\T)$ with $\supp(x)$ infinite. Define
\[
C = \supp(x) \setminus \supp^q(x), \qquad
D = \supp(x) \setminus (\supp(x)-1), \qquad
A = C \cup D.
\]
Since $\supp(x)$ is infinite and $\supp^q(x)$ cannot be cofinite, it follows that $A$ is infinite.

\noindent{Case 1.} $C$ is infinite. Since $\iI$ is an $snt$-ideal, by Definition~\ref{defsnt} there exists $m_1 \in \N\cup\{0\}$ such that
\[
B_1 = C - 1 - m_1 \notin \iI.
\]
For every $n \in B_1$, we have $n+m_1+1 \in \supp(x)$ and $n+m_1+1 \notin \supp^q(x)$. Hence, for all $n \in B_1$:

\[
\{a_nx\} = a_n \sum_{k=n+1}^{\infty} \frac{c_k}{a_k}
   \geq a_n \sum_{k=n+m_1+1}^{\infty} \frac{c_k}{a_k}
   \geq \frac{a_n}{a_{n+m_1+1}} \geq \frac{1}{M^{m_1+1}}.
\]

On the other hand, we also have
\begin{align*}
\{a_nx\}
&\leq a_n \left[\frac{(a_{n+1}/a_n - 1)}{a_{n+1}} + \cdots + \frac{(a_{n+m_1+1}/a_{n+m_1} - 2)}{a_{n+m_1+1}}\right]
   + a_n \sum_{k=n+m_1+2}^\infty \frac{c_k}{a_k} \\[6pt]
&\leq 1 - \frac{a_n}{a_{n+m_1+1}} \\[6pt]
&\leq 1 - \frac{1}{M^{m_1+1}}.
\end{align*}

Thus, there exists $B_1 \in \iI^+$ such that for all $n \in B_1$ we have
\[
\frac{1}{M^{m_1+1}} \leq \{a_nx\} \leq 1 - \frac{1}{M^{m_1+1}},
\]
contradicting $x \in t^\iI_{(a_n)}(\T)$.

\noindent{Case 2.} $D$ is infinite. Again, there exists $m_2 \in \N\cup\{0\}$ such that
\[
B_2 = D - 1 - m_2 \notin \iI.
\]
For all $n \in B_2$, we have $n+m_2+1 \in \supp(x)$ and $n+m_2+2 \notin \supp(x)$. Therefore, for each $n \in B_2$:

\[
\{a_nx\} = a_n \sum_{k=n+1}^{\infty} \frac{c_k}{a_k}
   \geq a_n \sum_{k=n+m_2+1}^{\infty} \frac{c_k}{a_k}
   \geq \frac{1}{M^{m_2+1}}.
\]

On the other hand,
\begin{align*}
\{a_nx\}
&\leq a_n \left[\frac{(a_{n+1}/a_n - 1)}{a_{n+1}} + \cdots + \frac{(a_{n+m_2+1}/a_{n+m_2} - 1)}{a_{n+m_2+1}}\right]
   + a_n \sum_{k=n+m_2+3}^\infty \frac{c_k}{a_k} \\[6pt]
&\leq \Bigl(1 - \frac{a_n}{a_{n+m_2+1}}\Bigr) + \frac{a_n}{a_{n+m_2+2}} \\[6pt]
&\leq 1 - \tfrac{1}{2}\cdot \frac{1}{M^{m_2+1}}.
\end{align*}

Hence, there exists $B_2 \in \iI^+$ such that for all $n \in B_2$ we have
\[
\frac{1}{M^{m_2+1}} \leq \{a_nx\} \leq 1 - \tfrac{1}{2}\cdot \frac{1}{M^{m_2+1}},
\]
which again contradicts $x \in t^\iI_{(a_n)}(\T)$. Since both cases lead to contradictions, it follows that $\supp(x)$ cannot be infinite. Thus, every $x \in t^\iI_{(a_n)}(\T)$ must have finite support, implying $x \in t_{(a_n)}(\T)$. Therefore, $t^\iI_{(a_n)}(\T) = t_{(a_n)}(\T)$.
%\qed
\end{proof}
We now look at the notion of the ideals belonging to $\mathcal{H}$. By definition, $\mathcal{H}$ contains the class of all dense translation-invariant ideals, and in particular the modular simple density ideals (for the definition see \cite{DG}). Our next result shows that the class $\mathcal{H}$, in fact, coincides with the class of dense ideals for $P$-ideals.
%%%%%%%%%%%%%%%%%%%%%%%%%%%%%%%%%%%%%%%%%%%%%%%%%%%%%%%%%%%%%%%%%%%%%%%%%%%%%%%%%%
\begin{theorem}\label{thdense01}
A $P$-ideal $\iI$ is dense if and only if it satisfies the property $(\natural)$.
\end{theorem}
\begin{proof}
First we assume that $\iI$ satisfies the property $(\natural)$. Consider any infinite $A\subseteq\N$. Then the property $(\natural)$ ensures that there exists an infinite $B\subseteq A$ such that $B-m\in\iI$ for each $m\in\Z$. In particular, we obtain that $B\in\iI$. Thus, we conclude that $\iI$ is dense.

Conversely, assume that $\iI$ is a dense $P$-ideal. Consider any infinite $A\subseteq\N$. Since, $\iI$ is dense there exists a lacunary $B_0\subseteq A$ such that $B_0\in\iI$. If $B_0$ is $\iI$-translation invariant then we are done here. So, we assume that $B_0$ is not $\iI$-translation invariant, i.e., there exists $m_0\in\Z$ such that $B_0-m_0\not\in\iI$ and $B_0-t\in\iI$ for $-|m_0|<t<|m_0|$. Since $\iI$ is dense again we have $B_1'\subseteq (B_0-m_0)$ such that $B_1'\in\iI$. If $B_1'+2m_0\in\iI$ then we set $B_1=B_1'+m_0$. If $B_1'+2m_0\not\in\iI$, there exists $B_1''\subseteq B_1'+2m_0$ such that $B_1''\in\iI$ and we set $B_1=B_1''-m_0$. Consequently, we obtain that
$$
B_1\subseteq B_0 \mbox{ and } \bigcup\limits_{t=-|m_0|}^{|m_0|} (B_1-t) \in\iI.
$$
If $B_1$ is $\iI$-translation invariant then we are done here. So, we assume that $B_1$ is not $\iI$-translation invariant, i.e., there exists $m_1\in\Z\setminus [-|m_0|,|m_0|]$ such that $B_1-m_1\not\in\iI$ and $B_1-t\in\iI$ for $-|m_1|<t<|m_1|$. Then proceeding exactly same as above we conclude that there exists $B_2\subseteq B_1$ such that
$$
\bigcup\limits_{t=-|m_1|}^{|m_1|} (B_2-t) \in\iI.
$$
Now, for each $i\in\N$, if $B_i$ is not $\iI$-translation invariant there exists $m_i\in\Z\setminus [-|m_{i-1}|,|m_{i-1}|]$ such that
$$
B_i-m_i\not\in\iI, \ B_i\subseteq B_{i-1} \mbox{ and } \bigcup\limits_{t=-|m_{i-1}|}^{|m_{i-1}|} (B_i-t) \in\iI.
$$
For each $i\in\N$, we set
$$
D_i= \bigcup\limits_{t=-|m_{i-1}|}^{|m_{i-1}|} (B_i-t) \in\iI.
$$
Since $\iI$ is a $P$-ideal, there exists $D\in\iI$ such that $D_i\subseteq^\ast D$ for each $i\in\N$. We set
$$
B=\{n_k\in B_1: \bigcup\limits_{i=0}^k [n_k-i,n_k+i]\in D\}.
$$
Since $B_i\subseteq B_j$ for each $i\geq j$ and $D_i\subseteq^\ast D$, we conclude that $B$ is infinite. Now, for each $m\in\Z$, observe that $B-m\subseteq^\ast D$. Since $D\in\iI$, we obtain that $B\subseteq A$ is $\iI$-translation invariant. Thus, $\iI$ satisfies the property $(\natural)$.
\end{proof}
%%%%%%%%%%%%%%%%%%%%%%%%%%%%%%%%%%%%%%%%%%%%%%%%%%%%%%%%%%%%%%%%%%%%%%%%%%%%%
Observe that if $\iI$ is an $snt$ ideal, then necessarily $\iI \notin \mathcal{H}$, i.e., the the class $\mathcal{H}$ is contained in the class of non-$snt$ ideals. Our next result shows that the class of non-$snt$ ideals strictly contains the class of dense $P$-ideals.
\begin{proposition}\label{p04}
There exists a non-$snt$ $P$-ideal $\iI$ which is not dense.
\end{proposition}
\begin{proof}
Consider two sets $B_1={(n_i)},\ B_2={(m_i)}$ in $\N$ such that
$$
n_i<m_i<n_{i+1},\ (m_i-n_i)\rightarrow \infty\  \&  \  (n_{i+1}-m_i)\rightarrow\infty.
$$
We define a non-negative regular matrix $\textbf{M}={(a_{ij})}$ where,
$$ a_{ij}=
 \begin{cases}
   1 & \text{if}\ j= n_i \\
   0 & \text{otherwise}.
\end{cases}
$$
Now consider any infinite $A\subseteq B_1$. Then we can write $A=(n_{i_k})$ and $B=\{i_k:k\in\N\}$. Now observe that
$$
\overline{d}_{\textbf{M}}(A)=\limsup\limits_{n\to\infty}\big(\textbf{M}\cdot\chi(A)\big)=\limsup\limits_{n\to\infty} \chi(B)=1,
$$
Therefore, we conclude that $\iI$ is not dense.

In order to show that $\iI$ is non-$snt$ consider the infinite set $A=B_2$. From our construction, $A\subset B_1^{\textbf{c}}\  \Rightarrow\  \ d_{\textbf{M}}(A)\leq d_{\textbf{M}}(B_1^{\textbf{c}})=0$, i.e., $d_{\textbf{M}}(A)=0$.
Choose any $k\in\Z$. Then there exists $i_0\in\N$ such that $ m_i-n_i > |k|$ and $ n_{i+1}-m_i > |k|$ for all $i\geq i_0$. Consequently we obtain
$$
(A-k)\cap\ [m_{i_0},\infty) \subseteq B_1^{\textbf{c}}\ ~ \Rightarrow\ \ d_{\textbf{M}}(A)=d_{\textbf{M}}(A-k)=0 \ ~ \mbox{ for all } \ ~ \ k\in\Z
$$
which shows that every translation of $A$ belongs to $\iI$ concluding the proof.
\end{proof}
%%%%%%%%%%%%%%%%%%%%%%%%%%%%%%%%%%%%%%%%%%%%%%%%%%%%%%%%%%%%%%%%%%%%%%%%%%%%%%%%%%%%%%%%
Our next result shows that $\mathcal{H}$ is not contained in the class of dense translation invariant $P$-ideals, i.e., the class of dense $P$-ideals properly contains the class of dense translation invariant $P$-ideals.
\begin{proposition}\label{p03}
There exists a dense $P$-ideal $\iI$ which is not translation invariant.
\end{proposition}
\begin{proof}
Let us write $S=\{n^2 \ : \ n\in\N\}$. Consider the non-negative regular matrix $\textbf{M}$ defined by
$$
\textbf{M}=(a_{ij}) \ \mbox{ and }\ \ a_{ij}=
  \begin{cases}
   \frac{1}{i} & \text{if}\ j\in S \ \text{and}\ j\leq i^2,\\
   0 & \text{otherwise}.
   \end{cases}
$$
Then corresponding $\textbf{M}$-density ideal is defined as
$$
\iI= \{A\subset\N : d_{\textbf{M}}(A)=0 \}.
$$
Now, observe that
$$
\overline{d}_{\textbf{M}}(S)=\limsup\limits_{n\to\infty}\big(\textbf{M}\cdot\chi(S)\big)=\limsup\limits_{n\to\infty} \{1,1,1,...,1,....\}=1,
$$
and,
$$
\overline{d}_{\textbf{M}}(\N\setminus S)=\limsup\limits_{n\to\infty}\big(\textbf{M}\cdot\chi(S)\big)=\limsup\limits_{n\to\infty} \{0,0,0,...,0,....\}=0.
$$
Since, $S+1\subseteq \N\setminus S$, we must have $d_{\textbf{M}}(S+1)=0$. Therefore, $\iI$ is not a translation invariant ideal.

Consider any $B\subseteq\N$ such that $\overline{d}_{\textbf{M}}(B)>0$. To show that $\iI$ is dense, it is enough to show that there exists $A\subseteq B$ such that $d_{\textbf{M}}(A)=0$. If $D=B\cap (\N\setminus S)$ is infinite, then $d_{\textbf{M}}(D)=0$ and $D\subseteq B$. So let us assume that $B\subseteq^\ast S$. We set
$$
C=\{m_k\in B : m_{k+1}\geq (k+1)^2 m_k\}
$$
Then observe that
\begin{eqnarray*}
\overline{d}_{\textbf{M}}(C) &=& \limsup\limits_{n\to\infty}\big(\textbf{M}\cdot\chi(S)\big)\\ &\leq& \limsup\limits_{n\to\infty} \{\frac{1}{m_1},\frac{1}{m_1+1},\frac{1}{m_1+2},...,\frac{1}{m_2},\frac{1}{m_2+1},...,\frac{1}{m_k},....\}=0.
\end{eqnarray*}
Therefore, we conclude that the ideal $\iI$ is dense.
\end{proof}

%%%%%%%%%%%%%%%%%%%%%%%%%%%%%%%%%%%%%%%%%%%%%%%%%%%%%%%%%%%%%%%%%%%%%%%%%%%%%%%%%%%
The next lemma was established in \cite{DGT} to build a bridge between an arbitrary arithmetic sequence and an increasing sequence of naturals to form a potent tool to compare differently characterized subgroups with classical characterized subgroups or better Arbault sets which they generate and here again it will play a very prominent role to establish the main result of this section.

\begin{lemma}\label{lg1}\cite{DGT}
Let $(u_n)$ be an arithmetic sequence and $(a_n)$ be an increasing sequence of naturals. If $G=\{\frac{1}{u_n}:n\in\N\}\subseteq t_{(a_n)}(\T)$ then $a_n$ must be of the form $u_{k_n}v_n$ where $k_n\to\infty$ and $q_{k_n+1}$ does not divide $v_n$ for any $n\in\N$.
\end{lemma}

\begin{corollary}\label{l21}\cite{DGT}
Let $G=\{\frac{1}{p^n}  : n\in\N\}$ and $(a_n)$ be an increasing sequence of naturals. If $\ G\subseteq t_{(a_n)}(\T)$ then $a_n$ must be of the form $p^{k_n}v_n$ where $k_n\to\infty$ and $p\nmid v_n$.
\end{corollary}

\begin{proposition}\label{PIsupp}\cite[Lemma 3.2]{DG6}
For any analytic $P$-ideal $\iI$ and $x\in\T$, if $supp(x) \in\iI$ and $supp(x)$ is $\iI$-translation invariant, then $x\in t^{\iI}_{(a_n)}(\T)$.
\end{proposition}
%%%%%%%%%%%%%%%%%%%%%%%%%%%%%%%%%%%%%%%%%%%%%%%%%%%%%%%%%%%%%%%%%%%%%%%%%%%%%%%%%%%%%%%%%%%
The class of $\iI$-characterized subgroups was introduced and investigated for the class of arithmetic sequences in \cite{DG6}. It was shown that for non-$snt$ ideals, the $\iI$-characterized subgroup corresponding to an arithmetic sequence is always bigger than it usual counterpart (see \cite[Theorem 3.6]{DG6}). But this does not guarantee that this class of $\iI$-characterized subgroup is essentially new. Our next result sheds light in this direction and fills this gap. In particular, it says that, for the class of dense $P$-ideals, the $\iI$-characterized subgroup cannot be characterized by any sequence of integers which establish the novelty of the notion that was missing in \cite{DG6}.
\begin{theorem}\label{th6}
For any arithmetic sequence $(u_n)$ and $\iI\in\mathcal{H}$, the subgroup $t^{\iI}_{(u_n)}(\T)$ is not an $\AR$-set.
\end{theorem}

\begin{proof}
Let $(u_n)$ be an arithmetic sequence. If possible assume that there exists an increasing sequence of naturals $(a_n)$ such that
$$
t^{\iI}_{(u_n)}(\T)\subseteq t_{(a_n)}(\T).
$$
Observe that for each $x\in S=\{\frac{1}{u_n}:n\in\N\}$, $supp_{(u_n)} (x)$ is finite. Then Proposition \ref{PIsupp} ensures that $S\subseteq t^{\iI}_{(u_n)}(\T)$. Therefore, in view of Lemma \ref{lg1} it follows that $a_n$ can be written as $u_{k_n}v_n$ where $k_n\to\infty$ and $q_{k_n+1}$ does not divide $v_n$ for all $n\in\N$.

Since $\iI\in\mathcal{H}$, there exists an infinite subset $A$ of $\N$ such that $A\in\iI$ and $A$ is $\iI$-translation invariant and we can choose a subsequence $(a_{n_i})$ of $(a_n)$ (with $a_{n_1}=a_1$) such that
\begin{itemize}
\item[(i)] $u_{k_{n_{(i+1)}}}\geq 8a_{n_i}$,
\item[(ii)] $k_{n_i}\in A$.
\end{itemize}
Let us define an element $x\in\T$ in the following way.
\begin{equation}\label{suppeq}
\supp_{(u_n)}(x)=\{k_{n_i}+1 \ : \ i\in\N\} \ \mbox{    and   } c_r=\bigg\lfloor \frac{q_r}{m_r} \bigg\rfloor \ \mbox{ for all } r\in \supp_{(u_n)}(x).
\end{equation}
(where $1<m_r\leq q_r$ ). Now observe that

\begin{eqnarray*}
\big\{a_{n_i}x\big\} &=& \bigg\{v_{n_i}u_{k_{n_i}}\sum\limits_{r=k_{n_i}+1}^\infty\frac{c_r}{u_r}\bigg\} \\  &=& \bigg\{v_{n_i}c_{k_{n_i}+1}\frac{u_{k_{n_i}}}{ u_{k_{n_i}+1}}+v_{n_i}u_{k_{n_i}}\sum\limits_{r=k_{n_{(i+1)}}+1}^\infty\frac{c_r}{u_r}\bigg\}
\\
\Rightarrow \ \big\{a_{n_i}x\big\} &=& \bigg\{c_{k_{n_i}+1}\frac{v_{n_i}}{ q_{k_{n_i}+1}}+a_{n_i}\sum\limits_{r=k_{n_{(i+1)}}+1}^\infty\frac{c_r}{u_r}\bigg\}.
\end{eqnarray*}
Since $q_{k_{n_i}+1}\nmid v_{n_i}$ for all $i\in\N$, we must have
\begin{equation}\label{eqli}
\bigg\{\frac{v_{n_i}}{q_{k_{n_i}+1}}\bigg\} = \frac{l_i}{q_{k_{n_i}+1}} \mbox{   for some  } l_i\in\big\{1,2,\ldots,q_{k_{n_i}+1}-1\big\}.
\end{equation}
Next let us set
\begin{equation}\label{suppeq1}
l_i'=q_{k_{n_i}+1} - l_i \ \mbox{ and }  \  m_{k_{n_i}+1}=
  \begin{cases}
   2l_i & \text{if}\ l_i\leq \frac{q_{k_{n_i}+1}}{2}, \\
   2l_i' & \text{if}\ l_i > \frac{q_{k_{n_i}+1}}{2}.
   \end{cases}
\end{equation}
\begin{itemize}
\item[$\bullet$]  For $l_i\leq \frac{q_{k_{n_i}+1}}{2}$, from (\ref{suppeq1}) we can write
\begin{eqnarray*}
c_{k_{n_i+1}} &=& \frac{q_{k_{n_i+1}}-e_i}{2l_i} \mbox{ (say)      for some  } e_i\in\{1,2,\ldots, 2l_i-1\} \\
\Rightarrow \bigg\{c_{k_{n_i}+1}\frac{v_{n_i}}{ q_{k_{n_i}+1}}\bigg\} &=& \bigg\{c_{k_{n_i}+1}\frac{l_i}{q_{k_{n_i}+1}}\bigg\} \\  &=& \bigg\{\frac{l_i c_{k_{n_i}+1}}{e_i+2l_ic_{k_{n_i}+1}}\bigg\}.
\end{eqnarray*}
Therefore we can conclude that
$$
\frac{1}{4} \leq \bigg\{c_{k_{n_i}+1}\frac{v_{n_i}}{ q_{k_{n_i}+1}}\bigg\}\leq\frac{1}{2}.
$$

\item[$\bullet$]  On the other hand when  $l_i > \frac{q_{k_{n_i}+1}}{2}$, from (\ref{suppeq1}) we can write
\begin{eqnarray*}
c_{k_{n_i+1}} &=& \frac{q_{k_{n_i+1}}-e_i'}{2l_i'} \mbox{ (say)      for some  } e_i'\in\{1,2,\ldots, 2l_i'-1\} \\
\Rightarrow \bigg\{c_{k_{n_i}+1}\frac{v_{n_i}}{ q_{k_{n_i}+1}}\bigg\} &=& \bigg\{c_{k_{n_i}+1}\frac{l_i}{q_{k_{n_i}+1}}\bigg\} \\
 &=&  \bigg\{c_{k_{n_i}+1}\frac{q_{k_{n_i}+1}-l_i'}{q_{k_{n_i}+1}}\bigg\} \\   &=&  \bigg\{c_{k_{n_i}+1}-\frac{l_i'c_{k_{n_i}+1}}{q_{k_{n_i}+1}}\bigg\} \\ &=& \bigg\{c_{k_{n_i}+1}-\frac{l_i' c_{k_{n_i}+1}}{e_i'+2l_i'c_{k_{n_i}+1}}\bigg\}
\end{eqnarray*}
from which we can conclude that
$$
\frac{1}{2} \leq \bigg\{c_{k_{n_i}+1}\frac{v_{n_i}}{ q_{k_{n_i}+1}}\bigg\}\leq\frac{3}{4}.
$$
\end{itemize}
So for both cases we obtain that
$$
\frac{1}{4} \leq \bigg\{c_{k_{n_i}+1}\frac{v_{n_i}}{ q_{k_{n_i}+1}}\bigg\}\leq\frac{3}{4}.
$$
From Property (i), we also have
$$
0\leq a_{n_i}\sum\limits_{r=k_{n_{(i+1)}}+1}^\infty\frac{c_r}{u_r}\leq \frac{a_{n_i}}{u_{k_{n_{(i+1)}}}} \leq\frac{1}{8}.
$$
Combining everything we observe that
$$
 \frac{1}{4}\leq \{a_{n_i}x\}=\bigg\{c_{k_{n_i+1}}\frac{v_{n_i}}{ q_{k_{n_i+1}}} +a_{n_i}\sum\limits_{r=k_{n_{(i+1)}}+1}^\infty\frac{c_r}{u_r}\bigg\} \leq \frac{7}{8}.
$$
This shows that $x\not\in t_{(a_n)}(\T)$. As we also have $\supp_{(u_n)}(x)\subseteq A+1\in\iI$ so Proposition \ref{PIsupp} ensures that $x\in t^{\iI}_{(u_n)}(\T)$ which is a contradiction. Consequently we must have $t^{\iI}_{(u_n)}(\T)\nsubseteq t_{(a_n)}(\T)$. Since the collection of all characterized subgroups forms a basis of the family of Arbault sets $\A$, it immediately follows that $t^{\iI}_{(u_n)}(\T)$ is not an $\AR$-set.
\end{proof}

\section{ Ideal version of Arbault sets and certain observations in respect of the family $\iI\A$ for ideals in $\mathcal{H}$}

 We start with the axiomatic approach to define trigonometric thin sets in general parlance. Following \cite{BKR} we define a family $\F$ of subsets of $\T$ to be a family of thin sets if the following conditions hold:
\begin{itemize}
\item[ (a) ] For each $x\in\T$, $\{x\}\in\F$,
\item[ (b) ] if $Y\subseteq X \in\F$ then $Y\in\F$,
\item[ (c) ] $\F$ does not contain any open interval.
\end{itemize}

From the family $\F$ one may construct a new family $\F_\sigma$ by
$$
\F_\sigma=\big\{ \bigcup\limits_{n=1}^\infty X_n: X_n\in\F \big\}.
$$
%
%
%The typical small subsets of the unit interval $[0,1]$ are the meager sets (i.e. the sets of the first Baire category), negligible sets (i.e. the sets of Lebesgue measure zero) and the porous sets (for a complete description see \cite{Za}). Since these sets are going to play some interesting roll in this article, we denote
%\begin{eqnarray*}
%\mathcal{M} &=&\{X\subseteq [0,1]: X \mbox { is meager}\},\\
%\mathcal{L} &=& \{X\subseteq [0,1]: X \mbox { is negligible}\},\\
%\mathcal{P} &=& \{X\subseteq [0,1]: X \mbox { is porous}\}.
%\end{eqnarray*}
%
%A set $X\subseteq [0,1]$ is called an $\mathbf{H}$-set if there exists an increasing (by the word increasing it would always mean strictly increasing) sequence of integers $(k_n)$ and an interval $I$ such that $k_nX\cap I=\emptyset$. We denote the family of all $\mathbf{H}$-sets by $\mathcal{H}$. In \cite{BKR} the authors established that $\mathcal{H}_\sigma\subseteq \mathcal{M}\cap\mathcal{L}$. Also in \cite{Za}, it was shown that $\mathcal{H}\subseteq \mathcal{P}$. Combining these results one can conclude that
%\begin{equation}\label{en1}
%\mathcal{H}_\sigma\subseteq \mathcal{M}\cap\mathcal{L}\cap\mathcal{P}_\sigma .
%\end{equation}
%
A subfamily $\mathcal{G}\subseteq \F$ is called a basis for $\F$ if for any $X\in\F$ there exists $Y\in\mathcal{G}$ such that $X\subseteq Y$. If the basis $\mathcal{G}$ consists of $F_\sigma$-sets, Borel sets, etc., then the basis is called $F_\sigma$ basis, Borel basis, etc. %It is well known that the families $\mathcal{M},\ \mathcal{L}$ and $ \mathcal{P}_\sigma$ have a $F_\sigma$ basis, a $G_\delta$ basis and a $G_{\delta\sigma}$ basis respectively.\\

The arithmetic difference of two sets $A,B\subseteq \T$ is defined as
$$
A-B=\{z\in\T: z=x-y \mbox{ for some } x\in A,\ y\in B\}.
$$
A family of thin sets $\F$ is called trigonometric if for any $A\in\F$, $A-A\in\F$.\\

For a real $x$, we denote its fractional part by $\{x\}$ and $\|x\|$ the distance from the integer i.e. $\min\big\{\{x\},1-\{x\}\big\}$. It is well known that
$$
2\|x\| \leq |\sin\pi x| \leq \pi \|x\|.
$$

As is customary we will use the inequality by replacing $|\sin\pi x|$ and $\|x\|$ one by another, wherever required to serve our purpose. Let us now recall some definitions of classical thin sets which have been of much interest (some equivalent definitions can be found in \cite{BKR,BuBu1,E1,Ka1}).\\

A set $X\subseteq [0,1]$ is called
\begin{itemize}
\item [ (1) ] a $\DD$-set (Dirichlet set) if there exists an increasing sequence of naturals $(a_n)$ such that $\|a_nx\|$ converges uniformly to 0 on $X$.
\item [ (2) ] an $\AR$-set (Arbault set) if there exists an increasing sequence of naturals $(a_n)$ such that $\|a_nx\|$ converges pointwise to 0 on $X$.
%\item [ (3) ] an $\mathbf{R}$-set if there exists a sequence of naturals $(a_n)$ and a sequence of reals $(r_n)$ such that $r_n \nrightarrow 0$ and $\sum\limits_{n=1}^\infty r_n \|a_nx\| < \infty$ for all $x\in X$.
\item [ (3) ] an $\NN$-set if there exists a sequence of naturals $(a_n)$ and a sequence of non-negative reals $(r_n)$ such that $\sum\limits_{n=1}^\infty r_n=\infty$ and $\sum\limits_{n=1}^\infty r_n \|a_nx\| < \infty$ for all $x\in X$.
\item [ (4) ] an $\NZ$-set if there exists a sequence of naturals $(a_n)$ such that $\sum\limits_{n=1}^\infty \|a_nx\| < \infty$ for all $x\in X$.
\item [ (5) ] a $\wD$-set (weak Dirichlet set) if $X$ is universally measurable and for every positive Borel measure $\mu$ on $[0,1]$ there exists an increasing sequence of naturals $(a_n)$ such that
    $$
    \lim\limits_{n\to\infty} \int_{X} |e^{2\pi i a_nx}-1|d\mu(x)=0.
    $$
    \\
\end{itemize}

The family of all $\DD$-sets, $\NN$-sets, $\NZ$-sets, $\AR$-sets, %$\mathbf{R}$-sets,
$\wD$-sets will be denoted by $\D$, $\fN$, $\fN_0$, $\A$, %$\mathcal{R}$,
$w\D$ respectively. It is known that $\D$, $\fN$, $\fN_0$, $\A$ are the typical examples of families of trigonometric thin sets with a Borel basis (see \cite{Ar}). There are well known results about these classical families of trigonometric thin sets (for exact references and proofs we refer to the survey paper \cite{BKR}, see also \cite{Ar,BaS,BuBu1,Ka,Ko,R} as also \cite{DGT}):
\begin{eqnarray*}
 &(R1)& \D\subsetneq \fN_0\subsetneq \A=\mathcal{R}\subsetneq w\D , \\
 &(R2)& \D\subsetneq \fN_0\subsetneq \fN\subsetneq w\D , \\
% &(R3)& \A\subsetneq \mathcal{H}_\sigma\subsetneq \mathcal{P}_\sigma , \\
% &(R4)& \fN\nsubseteq \mathcal{P}_\sigma , \\
 %&(R5)& \A\nsubseteq \fN_\sigma.
\end{eqnarray*}
%\\
Following the line of \cite{DGT} it is natural to introduce the notion of a $\iI$-Arbault set (in short $\iAR$-set) which is our prime interest in this article.
\begin{definition}
A set $X\subseteq [0,1]$ is called a $\iI$-Arbault set ($\iAR$-set in short) if there exists an increasing sequence of naturals $(a_n)$ such that $\|a_nx\|$ $\iI$-converges to 0 for all $x\in X$.
\end{definition}

Throughout, the family of all $\iAR$-sets will be denoted by $\iA$.
We start with the known observation that the collection of all characterized subgroups form a $F_{\sigma\delta}$ basis of the family $\A$ and also the collection of all $s$-characterized subgroups form a $F_{\sigma\delta}$ basis of the family $s\A$. Along the same line we have the following.

\begin{proposition}\label{p1}
The family $\iA$ has a $F_{\sigma\delta}$ basis consisting of $\iI$-characterized subgroups of $\T$.
\end{proposition}

\begin{proof}
Let $\mathcal{G}$ denote the family of all $\iI$-characterized subgroups of the circle, i.e.,
$$
\mathcal{G}=\{t^{\iI}_{(a_n)}(\T): \ (a_n) \mbox{ is an increasing sequence of naturals}\}.
$$
Since every member $t^{\iI}_{(a_n)}(\T)$ of the family $\mathcal{G}$ is an $\iAR$-set by definition of an $\iI$-characterized subgroup, it is obvious that $\mathcal{G}$ is a subfamily of $\iA$.

Now consider any $A\in \iA$. Then there exists an increasing sequence of naturals $(b_n)$ such that $\|b_nx\|\to 0$ with respect to $\iI$ for all $x\in A$. Therefore, $A\subseteq t^{\iI}_{(b_n)}(\T)$ and we conclude that $\mathcal{G}$ is a basis for the family $\iA$. In \cite{DG6} we have shown that every $\iI$-characterized subgroup is an $F_{\sigma\delta}$ subset of $\T$. Thus $\mathcal{G}$ is a $F_{\sigma\delta}$ basis for the family $\iA$.
\end{proof}
%The following lemma has been presented in \cite{Ba} which is a generalization of classical Dirichlet-Minkowski Theorem.
%\begin{lemma}\label{l1}
%Let $(a_n)$ denote an increasing sequence of naturals. For any $\varepsilon > 0$ and any reals $x_1,\ldots, x_k$, there are $i,j$ such that $0\leq i <
%j < \big(\frac{1}{\varepsilon}\big)^k$ and
% $$
% \|(a_i- a_j)x_l\|<2\varepsilon \mbox{ for } l = 1, 2, \ldots , k.
% $$
%\end{lemma}
%
%
%\begin{corollary}\label{c1}
%For any increasing sequence $(a_n)$ of naturals and any $x\in [0,1]$, there exists a subsequence $(a_{n_k})$ such that the sequence $(a_{n_k}x)$ converges.
%\end{corollary}
The following results immediately follows in line of \cite{DGT}.

\begin{proposition}\label{p2}
The family $\iA$ is a family of trigonometric thin sets.
\end{proposition}

%\begin{proof}
%Consider any $x\in\T$. Since every countable set is an Arbault set (so a $\iI$-Arbault set as well) by the main theorem of %\cite{Bo2}, we conclude that $\{x\}\in \iA$. If $A\subseteq B$ and $B\in \iA$, it is easy to observe that $A\in \iA$.

%Take any $A\in \iA$. Then from Proposition \ref{p1} there exists a $\iI$-characterized subgroup $H$ such that $A\subseteq H$. It is known that $\mu(H)=0$ (see \cite{DG}) from which one can conclude that the family $\iA$ cannot contain an open interval. Also observe that
%$$A- A \subseteq H- H=H\in \iA.$$
%Thus $\iA$ is a family of trigonometric thin sets.
%\end{proof}

\begin{proposition}\label{p4}
If $A\in \iA$ and $G$ is a subgroup of $\T$ generated by $A$, then $G\in \iA$.
\end{proposition}

%\begin{proof}
%Let $A\in \iA$ and $G$ be the subgroup of $\T$ generated by $A$. In view of Proposition \ref{p1} there exists a $\iI$-characterized subgroup $t^{\iI}_{(a_n)}(\T)$ of $\T$ containing $A$. But from the definition of $G$, we must have $A\subseteq G\subseteq t^{\iI}_{(a_n)}(\T)$. Since the family $\iA$ is a family of thin sets, we conclude that $G\in s\A$.
%\end{proof}

%\begin{lemma}\label{l2}\cite{Bk3}
%Let $\mathcal{F}$ be a family of trigonometric thin sets such that $\D\subseteq \mathcal{F}$. Then any base of $\mathcal{F}$ has cardinality at least $\mathfrak{c}$.
%\end{lemma}

\begin{proposition}\label{p3}
For the family $\iA$, the following hold:
\begin{itemize}
 \item[(i)] It cannot have a $F_{\sigma}$ basis,
 \item[(ii)] Every basis of $\iA$ has cardinality at least $\mathfrak{c}$.
\end{itemize}
\end{proposition}

%\begin{proof}
%(i) Since usual convergence implies $\iI$-convergence it is easy to observe that every $\AR$-set is an $\iAR$-set i.e. $\A \subseteq \iA$. Note that $\A$ does not have a $F_\sigma$ basis (for example the characterized subgroup $t_{2^{2^n}}(\T)$ which is clearly an $\AR$-set, cannot be contained in a $F_\sigma$ subset of $\T$ \cite{Ar}). Therefore, the family $\iA$ cannot have a $F_\sigma$ basis.\\

%(ii) Since $\D\subseteq\A\subseteq \iA$, the result follows directly from Lemma \ref{l2} .
%\end{proof}

%Further as expected, one can easily observe that the family $\iA$ is not an ideal.

%\begin{lemma}\label{l3}\cite{Bk3}
%Let $\mathcal{F}$ be a family of trigonometric thin sets. Then every member of $\mathcal{F}$ is meager and has Lebesgue measure zero.
%\end{lemma}

\begin{corollary}\label{c2}
$\iA\subseteq \mathcal{M} \cap \mathcal{L}$ where $\mathcal{M}$ and $\mathcal{L}$ denote respectively the collection of meagre sets and sets with Lebesgue measure zero.
\end{corollary}

From this point onward, all ideals under consideration are assumed to belong to the class $\mathcal{H}$. Our next two theorems ensure that the newly introduced  family $\iA$ is really new compared to the already investigated families of classical trigonometric thin sets (such as $\D,\mathcal{N}_0,\A,\mathcal{N}$ etc.) and provide a clear view regarding the families $\iA$, $\A$ and $\mathcal{N}$.

\begin{theorem}\label{th1}
For any arithmetic sequence $(u_n)$, the subgroup $t^{\iI}_{(u_n)}(\T)$ contains a $\NN$-set which is not an $\AR$-set.
\end{theorem}

\begin{proof}
Let us define
$$
A=\{x\in t^{\iI}_{(u_n)}(\T): \sum \frac{1}{n}|\sin u_n\pi x|<\infty\}.
$$
Considering the sequence $(r_n)$ where $r_n=\frac{1}{n}$ for all $n\in\N$, from definition it follows that $A$ is a $\NN$-set. Therefore, only fact left to be shown is that $A\not\in\A$.

Since the class of characterized subgroups forms a basis of $\A$ \cite{BDBW} so if possible assume that $A\subseteq t_{(a_n)}(\T)$ for some sequence of naturals $(a_n)$. It is easy to observe that $A$ contains $G=\{\frac{1}{u_n}:n\in\N\}$. Then Lemma \ref{lg1} ensures that the sequence $(a_n)$ will be of the form $a_n=u_{k_n}v_n$ where $k_n\to\infty$ and $q_{k_n+1}\nmid v_n$ for each $n\in\N$. We are going to show that we can always construct an element $x\in A$ depending on the sequence $(a_n)$ such that $x\not\in t_{(a_n)}(\T)$.

Since $\iI\in\mathcal{H}$, there exists an infinite subset $B$ of $\N$ such that $B\in\iI$ and $B$ is $\iI$-translation invariant and we can choose a subsequence $(a_{n_i})$ of $(a_n)$ (with $a_{n_1}=a_1$) such that
\begin{itemize}
\item[(i)] $k_{n_i}\in B$,
\item[(ii)]  $u_{k_{n_{(i+1)}}}\geq 8a_{n_i}$,
\item[(iii)] $\sum\limits_{i=1}^\infty\frac{1}{k_{n_i}}<\infty$.
\end{itemize}

Consider an element $x$ of $\T$ such that
\begin{equation}\label{suppeq2}
\supp_{(u_n)}(x)=\{k_{n_i}+1 \ : \ i\in\N\} \ \mbox{    and   } c_r=\bigg\lfloor \frac{q_r}{m_r} \bigg\rfloor \ \mbox{ for all } r\in \supp_{(u_n)}(x),
\end{equation}
where $1<m_r\leq q_r$, i.e., $x=\sum\limits_{r\in \supp_{(u_n)}(x)} \frac{c_r}{u_r}$. Now proceeding exactly as in Theorem \ref{th6}, we obtain that
$$
\Rightarrow\ \frac{1}{4}\leq \{a_{n_i}x\}=\bigg\{c_{k_{n_i+1}}\frac{v_{n_i}}{ q_{k_{n_i+1}}} +a_{n_i}\sum\limits_{r=k_{n_{(i+1)}}+1}^\infty\frac{c_r}{u_r}\bigg\} \leq \frac{7}{8}.
$$
Therefore, we get $x\not\in t_{(a_n)}(\T)$.\\

Since $k_{n_i}\in B$ and $B$ is $\iI$-translation invariant then we must have $\supp_{(u_n)}(x)\in\iI$. Consequently, Proposition \ref{PIsupp} ensures that $x\in t^{\iI}_{(u_n)}(\T)$. Let $j$ be any integer such that $k_{n_{(i-1)}}< j\leq k_{n_i}$. Note that
\begin{eqnarray*}
\{u_jx\}=u_j\sum\limits_{r=i}^\infty\frac{c_r}{u_{k_{n_r}+1}} &\leq& \frac{u_j}{u_{k_{n_i}}} \\
\Rightarrow |\sin u_j\pi x| \leq \sin \frac{\pi u_j}{u_{k_{n_i}}} &\leq& \frac{\pi u_j}{u_{k_{n_i}}}
\end{eqnarray*}
\begin{eqnarray*}
\Rightarrow \sum\limits_{j=k_{n_{(i-1)}}+1}^{k_{n_i}} \frac{|\sin u_j\pi x|}{j} &<& \pi \sum\limits_{j=k_{n_{(i-1)}}+1}^{k_{n_i}} \frac{u_j}{j u_{k_{n_i}}} \\ &<& \frac{\pi}{k_{n_{(i-1)}}} \sum\limits_{j=k_{n_{(i-1)}}+1}^{k_{n_i}} \frac{u_j}{u_{k_{n_i}}} < \frac{2\pi}{k_{n_{(i-1)}}}.
\end{eqnarray*}
Consequently we have $\sum \frac{|\sin u_n\pi x|}{n}\leq \sum\limits_{i=1}^\infty \frac{2\pi}{k_{n_i}}<\infty$ and so $x\in A$ $-$which is a contradiction since $A\subseteq t_{(a_n)}(\T)$ and we have already shown that $x\not\in t_{(a_n)}(\T)$. Thus, there does not exist any increasing sequence $(a_n)$ for which $A\subseteq t_{(a_n)}(\T)$. Therefore, $A\not\in \A$.
\end{proof}

\begin{corollary}\label{c01}
$ \iA \cap \fN \nsubseteq\A$.
\end{corollary}
\begin{proof}
The assertion follows directly from Theorem \ref{th1}.
\end{proof}

\begin{corollary}\label{c3}
$\A\subsetneq \iA$.
\end{corollary}

\begin{proof}
The assertion follows directly from Theorem \ref{th6}.
\end{proof}

\begin{corollary}\label{c9}
$\iA\nsubseteq \mathcal{N}_\sigma$.
\end{corollary}

\begin{proof}
Since $\A\nsubseteq \mathcal{N}_\sigma$ \cite{Ka}, the result follows directly from Corollary \ref{c3}.\\
\end{proof}

\begin{theorem}\label{th3}
$\iA\nsubseteq \mathcal{N}\cup\A$.
\end{theorem}

\begin{proof}
In order to establish the result we are going to find an $A\in \iA\setminus \A$ which cannot be contained in an $F_{\sigma}$ set which in turn would imply that $A$ cannot be an $\NN$-set since the family $\mathcal{N}$ has a $F_\sigma$-basis.

Consider $A=t^{\iI}_{(2^{2^n})}(\T)$. Then Theorem \ref{th6} ensures that $A\in\iA\setminus\A$. Arbault had already shown that the set $t_{(2^{2^n})}(\T)$ cannot be contained in any $F_\sigma$ set \cite{Ar}. Since $t_{(2^{2^n})}(\T)\subsetneq t^\iI_{(2^{2^n})}(\T)=A$ (by \cite[Theorem 3.6]{DG6}), we can conclude that $A$ also cannot be contained in any $F_\sigma$ set. Therefore, $A\in \iA\setminus \mathcal{N}$.
\end{proof}

It is well known that a subfamily of $\mathcal{N}$, namely $\mathcal{N}_0$ is contained in $\A$. When it comes to the larger class $\iA$, one can again find a suitable subfamily of $\mathcal{N}$ which we denote by $\mathcal{N}^\iI$ (containing $\mathcal{N}_0$) which is contained in $\iA$.\\
\begin{definition}
$X\subseteq [0,1]$ is in $\mathcal{N}^\iI$ if there exists an increasing sequence of naturals $(a_n)$ and a decreasing sequence of positive reals $(r_n)$ with $\sum\limits_{n=1}^\infty r_n = \infty$ and $\sum\limits_{n\in A} r_n < \infty$ implies $A\in\iI$ such that $\sum\limits_{n=1}^\infty r_n \|a_nx\| < \infty$ for all $x\in X$.
\end{definition}

\begin{remark} The class $\mathcal{N}^\iI$ happens to contain certain members of the family $\mathcal{N}$, in particular $\NN$-sets constructed for counterexamples. In \cite{Ar}, in order to construct a $\NN$-set which is not an $\mathbf{A}$-set, Arbault had chosen the set $O_{(2^n)}(\T)$. Again the $\mathbf{N}'$-set $O_{(2^n)}(\T)$ was used in \cite{Ka} to give an example of a $\mathbf{N}$-set which is not a $L_0^\sigma$-set. The $\mathbf{N}'$-set $O_{(n!)}(\T)$ is a $\mathbf{N}$-set which is not $\sigma$-porus \cite{Ko} etc. (for details see also \cite{DGT}).
\end{remark}

\begin{proposition}
$\mathcal{N}^\iI\subseteq \mathcal{N}$.
\end{proposition}

\begin{corollary}
$\A\nsubseteq \mathcal{N}^\iI$.
\end{corollary}

\begin{theorem}\label{th2}
$\mathcal{N}^\iI\subsetneq \iA$.
\end{theorem}

\begin{proof}
Let $B\in \mathcal{N}^\iI$. Then there exists an increasing sequence of naturals $(a_n)$ and a decreasing sequence of positive reals $(r_n)$ with $\sum\limits_{n=1}^\infty r_n = \infty$ and $\sum\limits_{n\in A} r_n < \infty$ implies $A\in\iI$ such that $\sum\limits_{n=1}^\infty r_n \|a_nx\| < \infty$ for all $x\in B$. First we show that $B\subseteq t^\iI_{(a_n)}(\T)$.

Take any $x\in B$. From the fact $\sum\limits_{n=1}^\infty r_n \|a_nx\| < \infty$ it readily follows that $\lim\limits_{n\to\infty} r_n \|a_nx\|=0$. If $\lim\limits_{n\to\infty} \|a_nx\|=0$ then $x\in t_{(a_n)}(\T)\subseteq t^\iI_{(a_n)}(\T)$ and we are done. Now assume that $\lim\limits_{n\to\infty} \|a_nx\|\neq 0$. Consequently, there exists a subsequence $(a_{n_k})$ of $(a_n)$ such that $\|a_{n_k}x\|\geq \epsilon_0$ for some $\epsilon_0\in (0,\frac{1}{2}]$. \\

Observe that
\begin{equation}\label{eqN}
\epsilon_0\sum\limits_{k=1}^\infty r_{n_k}\ \leq \ \sum\limits_{n=1}^\infty r_{n_k}\|a_{n_k}x\| \ \leq \ \sum\limits_{n=1}^\infty r_{n}\|a_{n}x\| \ < \ \infty
\end{equation}
which shows that the series $\sum\limits_{k=1}^\infty r_{n_k}$ is convergent. Therefore we must have $$\{n_k: k\in\N\} \in\iI.$$

So we can conclude that there does not exist any subsequence $(a_{n_k})$ of $(a_n)$ with $\{n_k:k\in\N\}\not\in\iI$ for which $\|a_{n_k}x\|\geq \varepsilon$ for some $\varepsilon\in (0,\frac{1}{2}]$. Hence $(a_nx)$ must converge to 0 w.r.t. $\iI$ and so $x\in t^\iI_{(a_n)}(\T)$. Since $x$ was chosen arbitrarily, we obtain $B\subseteq t^\iI_{(a_n)}(\T)$. Therefore Proposition \ref{p1} entails that $B\in \iA$ and we get $\mathcal{N}^\iI\subseteq \iA$. Since $\mathcal{N}^\iI\subseteq \mathcal{N}$, the strictness of the inclusion follows from Theorem \ref{th3}.
\end{proof}

\begin{theorem}\label{th02}
$\mathcal{N}^\iI\nsubseteq \A$.
\end{theorem}

\begin{proof}
Let $(r_n)$ be a decreasing sequence of positive reals such that $\sum\limits_{n=1}^\infty r_n = \infty$ and $\sum\limits_{n\in A} r_n < \infty$ implies $A\in\iI$. We define
$$
B=\{x\in\T: \sum r_n|\sin p^n\pi x|<\infty\}.
$$
Clearly $B\in \mathcal{N}^\iI$. We are going to show that $B\not\in\A$.

Since the class of characterized subgroups forms a basis of $\A$ \cite{BDBW} so if possible assume that $B\subseteq t_{(a_n)}(\T)$ for some sequence of naturals $(a_n)$. It is easy to observe that $G=\{\frac{1}{p^n}:n\in\N\}\subseteq B$. Then Corollary \ref{l21} ensures that the sequence $(a_n)$ will be of the form $a_n=p^{k_n}u_n$ where $k_n\to\infty$ and $p\nmid u_n$. We are going to show that we can always construct an element $x\in B$ depending on the sequence $(a_n)$ such that $x\not\in t_{(a_n)}(\T)$.

Now, we choose a subsequence $(a_{n_i})$ of $(a_n)$ (with $a_{n_1}=a_1$) such that
\begin{itemize}
\item[(i)]  $k_{n_{i}}\geq k_{n_{(i-1)}}+(2n_{(i-1)}+1) u_{n_{(i-1)}}$,
\item[(ii)] $\sum\limits_{i=1}^\infty r_{k_{n_i}}<\infty$.
\end{itemize}

Consider an element $x$ of $\T$ such that
$$
\supp_{(p^n)}(x)=\{k_{n_i}+1: i\in\N\}
$$
and $c_t(x)=1$ for all $t\in \supp_{(p^n)}(x)$ i.e. $x=\sum\limits_{t\in \supp_{(p^n)}(x)} \frac{1}{p^t}$. Let $j$ be any integer such that $k_{n_{(i-1)}}< j\leq k_{n_i}$. Note that
$$
\{p^jx\}=p^j\sum\limits_{t=i}^\infty\frac{1}{p^{k_{n_t}+1}}\leq \frac{1}{p^{k_{n_i}-j}}
$$
$$
\Rightarrow |\sin p^j\pi x|\leq \sin \frac{\pi}{p^{k_{n_i}-j}}\leq \frac{\pi}{p^{k_{n_i}-j}}
$$
$$
\Rightarrow \sum\limits_{j=k_{n_{(i-1)}}+1}^{k_{n_i}} r_j|\sin p^j\pi x|< \pi \sum\limits_{j=k_{n_{(i-1)}}+1}^{k_{n_i}} r_j p^{k_{n_i}-j}< 2\pi r_{k_{n_{(i-1)}}}.
$$
Consequently we have $\sum r_n |\sin p^n\pi x|\leq 2\pi  \sum\limits_{i=1}^\infty r_{k_{n_i}} < \infty$ and so $x\in B$.
\\

Now observe that
$$
a_{n_i}x\equiv_{\Z} u_{n_i}\sum\limits_{t=1}^{\infty} \frac{c_{k_{n_i}+t}(x)}{p^t}
$$
$$
\Rightarrow \{a_{n_i}x\} =\bigg\{\frac{u_{n_i}}{p}\bigg\}+u_{n_i}\sum\limits_{t=b_i+1}^\infty \frac{c_{k_{n_i}+t}(x)}{p^t} \mbox{ where } b_i= k_{n_{(i+1)}}-k_{n_i}.
$$
Since $b_i>n_i u_{n_i}$ implies $p^{b_i}>n_i u_{n_i}$, we must have $\frac{u_{n_i}}{p^{b_i}}<\frac{1}{p^2}$. \\

So it is evident that
$$
\frac{1}{p}\leq \{a_{n_i}x\}\leq \frac{p-1}{p}+ \frac{u_{n_i}}{p^{b_i}}<1-\frac{p-1}{p^2}
$$
which implies $\|a_{n_i}x\|> \frac{p-1}{p^2}$. This shows that $x\not\in t_{(a_n)}(\T)$ $-$which is a contradiction since $B\subseteq t_{(a_n)}(\T)$ and we have already shown that $x\in B$. Thus, there does not exist any increasing sequence $(a_n)$ for which $B\subseteq t_{(a_n)}(\T)$. Therefore, $B\not\in \A$.
\end{proof}

Recall that an uncountable set $X$ is a Luzin set if every meager subset of $X$ is countable.

\begin{proposition}\label{p5}
$\iA \subseteq w\mathcal{D}$.
\end{proposition}

\begin{proof}
Take a $\iI$-characterized subgroup $t^\iI_{(a_n)}(\T)$ of $\T$. In \cite{Gh2} it has been shown that $t^\iI_{(a_n)}(\T)$ is closed and therefore a complete subgroup of $\T$ and $\mu(t^\iI_{(a_n)}(\T)\setminus t_{(a_n)}(\T))=0$. Therefore, $(e^{2\pi i a_nx})$ converges to $1$ $\mu$-almost every where on $t^\iI_{(a_n)}(\T)$. Since $|e^{2\pi i a_nx}|\leq 1$, in view of Dominated Convergence Theorem we observe that
$$
\lim\limits_{n\to \infty} \int_{t^\iI_{(a_n)}(\T)} |e^{2\pi i a_nx} -1|d\mu=0.
$$
Consequently $t^\iI_{(a_n)}(\T)$ is a $\wD$-set. Therefore, from Proposition \ref{p1} we conclude that $\iA \subseteq w\mathcal{D}$.
\end{proof}

\begin{corollary}\label{c09}
Under CH (the continuum hypothesis), $\iA \subsetneq w\mathcal{D}$.
\end{corollary}

\begin{proof}
Since every Luzin set which is non-meager while having strong measure zero is a $\wD$-set \cite{BKR}, Corollary \ref{c2} ensures that these sets do not belong to the family $\iA$. Thus we get $\iA \subsetneq w\mathcal{D}$.
\end{proof}

%Kunen \cite{Ku} had proved that there are no Luzin sets under the assumption of MA (Martin's Axiom) and the negation of CH. As in Corollary \ref{c9} the strict inclusion could only be obtained using Luzin sets, so a natural question arises whether one can obtain the result without using Luzin sets, i.e. without explicitly using CH.

%\begin{problem}
%Is $\iA \subsetneq w\mathcal{D}$ provable in ZFC?
%\end{problem}

The following three corollaries give a clearer picture about the class $\iA$ as it is seen that though the classes $\D$ and $\A$ are contained in $\iA$ the same do not remain true when countable unions of $\DD$-sets and $\AR$-sets come into consideration. Further it is also noted that the class $\fN_\sigma$ is definitely not contained in the class $\iA$.

\begin{proposition}\label{p7}\cite[Corollary 8.13]{BKR}
There are two perfect $\DD$-sets whose union is not a $\wD$-set. Consequently, $\D_\sigma \nsubseteq w\mathcal{D}$.
\end{proposition}

\begin{corollary}\label{c6}
$\D_\sigma \nsubseteq \iA$.
\end{corollary}

\begin{proof}
Follows directly from Proposition \ref{p5} and Proposition \ref{p7}.
\end{proof}

Our next two corollaries follow from Corollary \ref{c6} and from the fact that a $\DD$-set is an $\AR$-set as well as a $\NN$-set i.e. $\D_\sigma \subseteq \A_\sigma\cap \fN_\sigma$.

\begin{corollary}\label{c7}
$\A_\sigma \nsubseteq \iA$.
\end{corollary}

\begin{corollary}\label{c8}
$\fN_\sigma \nsubseteq \iA$.
\end{corollary}

\begin{proposition}\label{p6}\cite{BKR}
Every $F_\sigma$ $\wD$-set is an $\NN$-set.
\end{proposition}

\begin{corollary}\label{c5}
Every $F_\sigma$ $\iAR$-set is an $\NN$-set.
\end{corollary}

%%%%%%%%%%%%%%%%%%%%%%%%%%%%%%%%%%%%%%%%%%%%%%%%%%%%%%%%%%%%%%%%%%%%%%%%%%%%
\noindent{\textbf{Acknowledgement:}} \ The second author acknowledges support from DST(INSPIRE) through the Faculty Fellowship (DST/INSPIRE/04/2024/004079), during the tenure of which this work was undertaken.\\

%\noindent{\textbf{Authorship contribution statement:}} Pratulananda Das: Writing ? review and editing, Supervision, Ayan Ghosh: Writing ? original draft, Writing ? review and editing, Conceptualization.  \\

\noindent{\textbf{Conflict of Interest:}} The authors state that there is no conflict of interest. \\

\noindent{\textbf{Data Availability:}} Not applicable.
%%%%%%%%%%%%%%%%%%%%%%%%%%%%%%%%%%%%%%%%%%%%%%%%%%%%%%%%%%%%%%%%%%%%%%%%%%%%%

\end{document}